\documentclass[12pt]{amsart}
\usepackage{mathrsfs}
\usepackage{latexsym,amssymb,amsmath,amsfonts,graphicx,epstopdf,subfigure,float}
\usepackage{bbm}
\usepackage{pifont}
\usepackage{cite}
\usepackage{multirow}
\usepackage{appendix}

\def\dist{\mathop{\rm dist}\nolimits}

\def\diag{\mathop{\rm diag}\nolimits}
\def\diam{\mathop{\rm diam}\nolimits}

\newtheorem{thm}{Theorem}[section]
\newtheorem{lem}{Lemma}[section]
\newtheorem{prop}[lem]{Proposition}
\newtheorem{coro}[lem]{Corollary}

\newtheorem{defn}[lem]{Definition}

\numberwithin{equation}{section}
\newtheorem{remark}{Remark}[section]

\newcommand{\bi}{{\mathbf{i}}}
\newcommand{\bj}{{\mathbf{j}}}

\newcommand{\bD}{{\mathcal{D}}}
\newcommand{\SD}{{\mathcal{D}}}

\newcommand{\R}{{\mathbb{R}}}

\newcommand{\bomega}{{\boldsymbol {\omega}}}

 \usepackage{color}

\begin{document}
\title{Relations between topological and metrical properties of self-affine Sierpi$\acute{\text{N}}$ski sponges}

\author{Yuan Zhang} \address{Department of Mathematics and Statistics, Central China Normal University, Wuhan, 430079, China}
\email{yzhang@mail.ccnu.edu.cn}

\author{Liang-yi Huang$^*$ } \address{College of Mathematics and Statistics, Chongqing  University, Chongqing, 401331, China}
\email{liangyihuang@cqu.edu.cn}

\date{November 4, 2021}

\thanks {The work is supported by NSFS Nos. 11971195 and 11601172.}

\thanks{* The correspondence author.}

\maketitle

\begin{abstract}We investigate two Lipschitz invariants of metric spaces defined by $\delta$-connected components, called
the maximal power law property and the perfectly disconnectedness.
The first  property has been studied in  literature for some self-similar sets and Bedford-McMullen carpets,
while the second property seems to be new.
 For a self-affine Sierpi$\acute{\text{n}}$ski sponge $E$, we first show that $E$ satisfies the maximal power law if and only if $E$ and all its
major projections contain trivial connected components; secondly, we show that $E$ is perfectly disconnected if and only if
$E$ and all its major projections are totally disconnected.
\end{abstract}

\section{\textbf{Introduction}}
Let $d\geq 2$ and let
$ 2\leq n_1<n_2<\dots<n_d$ be a sequence of  integers. Let $\Lambda=\diag(n_1,\dots, n_d)$ be the $d\times d$ diagonal matrix. Let $\bD=\{\bi_1,\dots,\bi_N\}\subset\prod_{j=1}^d\{0,1,\dots,n_j-1\}$.
  For any $\bi\in\bD$ and $z\in\mathbb{R}^d$, we define $S_{\bi}(z)=\Lambda^{-1}(z+\bi)$, then $\{S_{\bi}\}_{\bi\in\bD}$ is an \emph{iterated function system} (IFS). The unique non-empty compact set $E=K(\{n_j\}_{j=1}^d,\bD)$ satisfying
\begin{equation}\label{Ssponge}
E=\bigcup_{\bi\in\bD}S_{\bi}(E)
\end{equation}
is called a $d$-dimensional \emph{self-affine Sierpi$\acute{\text{n}}$ski sponge}, see Kenyon and Peres \cite{KP96} and Olsen \cite{Olsen07}.
In particular, if $d=2$, then $E$ is called a \emph{Bedford-McMullen carpet}.

  There are a lot of works on dimensions, multifractal analysis
and other topics of    self-affine Sierpi$\acute{\text{n}}$ski sponges,  see for instance,
McMullen \cite{McMullen84}, Bedford \cite{Bedford84},   King \cite{King95},
Kenyon and Peres \cite{KP96},
  Olsen \cite{Olsen07}, Barral and Mensi  \cite{Bar07},
Jordan and Rams \cite{JR11}, Mackay \cite{MM11}, Fraser and Howroyd \cite{Fraser2017}.
Recently there are some works devoted to the Lipschitz classification of Bedford-McMullen carpets
(\cite{LLM13, MXX17,Rao2019,YZ20a}). The goal of the present paper is to search new Lipschitz invariants of self-affine
Sierpi$\acute{\text{n}}$ski sponges,  and investigate their relations with the topological properties
of the major projections of such sponges.

Let $(E, \rho)$ be a metric space. Let $\delta>0$. Two points $x,y\in E$ are said to be \emph{$\delta$-equivalent} if there exists a sequence $\{x_1=x,x_2,\dots,x_{k-1},x_k=y\}\subset E$ such that $\rho(x_{i},x_{i+1})\le\delta$ for $1\le i\le k-1$.
A $\delta$-equivalent class of $E$ is called a \emph{$\delta$-connected component} of $E$.
We denote by $h_E(\delta)$ the cardinality of the set of $\delta$-connected components of $E$.

Two positive sequences $\{a_i\}_{i\ge 1}$ and $\{b_i\}_{i\ge 1}$ are said to be \emph{comparable}, and denoted by $a_i\asymp b_i$, if there exists a constant $c>1$ such that $c^{-1}\le b_i/a_i\le c$ for all $i\geq 1$.
We define the maximal power law property as following.

\begin{defn}[Maximal power law]
\emph{Let $(E, \rho)$ be a compact metric space.
Let $\gamma>0$. We say $E$   satisfies the \emph{power law} with \emph{index} $\gamma$ if  $h_E(\delta)\asymp\delta^{-\gamma}$ ; if $\gamma=\dim_BE$ in addition, we say $E$ satisfies the \emph{maximal power law}.}
\end{defn}

The above definition is motivated by the notion of gap sequence.
 Gap sequence of a set on $\R$ was widely used by many mathematicians, for instance, \cite{BT54,K95,LP93}.
Using the function $h_E(\delta)$,  Rao, Ruan and Yang  \cite{RRY08} generalized the notion of gap sequence to
$E\subset \R^d$, denoted by $\{g_i(E)\}_{i\geq 1}$. It is shown in \cite{RRY08} that   if two compact subsets $E,E'\subset\mathbb{R}^d$ are Lipschitz equivalent,
then $g_i(E)\asymp g_i(E')$.
Actually, the definition and result in \cite{RRY08} are also valid for metric spaces, see  Section 2.

Miao, Xi and Xiong \cite{MXX17} observed that  $E$ satisfies the power law with index $\gamma$ if and only if   $g_i(E)\asymp i^{-1/\gamma}$ (see Lemma \ref{MXX}).
Consequently,  the (maximal) power law property is invariant under bi-Lipschitz maps.

Deng, Wang and Xi \cite{DWX15} proved that if $E\subset\mathbb{R}^d$ is a $C^{1+\alpha}$ conformal set satisfying the strong separation condition, then $E$  satisfies the maximal power law.
Miao \emph{et al.}  \cite{MXX17}  proved that a totally disconnected Bedford-McMullen carpet  satisfies the maximal power law  if and only if it possesses vacant rows.
 Liang, Miao and Ruan \cite{LMR2020} completely characterized the gap sequences of Bedford-McMullen carpets.
 For higher dimensional fractal cubes (see Section 3 for the definition), we show that

\begin{thm}\label{thm:cube} A fractal cube    satisfies the maximal power law if and only if  it has trivial points.
\end{thm}

Let $E$ be a self-affine Sierpi$\acute{\text{n}}$ski sponge defined in \eqref{Ssponge}.
A point $z\in E$ is called \emph{a trivial point} of $E$ if $\{z\}$ is a connected component of $E$.
For $x=(x_1,\dots, x_d)\in E$ and $1\leq j\leq d-1$, we set
$$\pi_j(x)=(x_{1},\dots, x_j).$$
We call $\pi_j(E)$ the $j$-th major projection of $E$.
We say $E$ is \emph{degenerated} if $E$ is contained in a face of $[0,1]^d$ with dimension $d-1$.

\begin{thm}\label{thm:law} Let $E$ be a non-degenerated self-affine Sierpi$\acute{\text{n}}$ski sponge defined in \eqref{Ssponge}. Then $E$
satisfies the maximal power law if and only if  $E$ and all $\pi_j(E)(1\le j\le d-1)$ possess trivial points.
\end{thm}

 The following definition characterizes a class of fractals that all $\delta$-connected components
 are small. Let  $\text{diam}\, U$ denote  the diameter of a set $U$.

\begin{defn}[Perfectly disconnectedness]
\emph{Let $(E,\rho)$ be a compact metric space. We say $E$ is \emph{perfectly disconnected},
if there is a constant $M_0>0$ such that for any $\delta$-connected component $U$ of $E$ with $0<\delta<\text{diam}(E)$,
  $\text{diam}\,U\leq M_0\delta$. }
\end{defn}

It is clear that  perfectly disconnectedness implies totally disconnectedness,
and the perfectly disconnectedness property is invariant under bi-Lipschitz maps.

\begin{remark}\label{perfectly_prop}
\emph{
 It is essentially shown in Xi and Xiong \cite{XX10} that
  a fractal cube is perfectly disconnected if and only if it is totally disconnected.
  We guess this may be true for a large class of self-similar sets.}
\end{remark}

\begin{thm}\label{thm:perfect}
Let $E$ be a non-degenerated self-affine Sierpi$\acute{\text{n}}$ski sponge defined in \eqref{Ssponge}. Then $E$ is
 perfectly disconnected   if and only if $E$ and all $\pi_j(E)(1\le j\le d-1)$  are totally
 disconnected.
\end{thm}

As a consequence of Theorem \ref{thm:law} and Theorem \ref{thm:perfect},
we obtain

\begin{coro}\label{Lip1}
Suppose $E$ and $E'$ are two non-degenerated  self-affine Sierpi$\acute{\text{n}}$ski sponges
in ${\mathbb R}^d$. If $E$ and $E'$ are Lipschitz equivalent, then
   \\
\indent\emph{(i)} if $E$ and all $\pi_j(E)$,  $1\leq j\leq d-1$ possess  trivial points, then so do $E'$ and $\pi_j(E')$,
  $1\le j\le d-1$;\\
\indent\emph{(ii)} if $E$ and all $\pi_j(E)$, $1\leq j\leq d-1$ are totally disconnected, so are $E'$ and
  $\pi_j(E')$, $1\leq j\leq d-1$.
\end{coro}

This paper is organized as follows: In Section 2, we give some basic facts about gap sequences of metric spaces.
Then we prove Theorem \ref{thm:cube}, Theorem \ref{thm:law} and Theroem \ref{thm:perfect} in Section 3, 4, 5 respectively.

\section{\textbf{Gap sequences of metric spaces}}
Let $(E,\rho)$ be a metric space. Recall that $h_E(\delta)$ is the cardinality of the set of $\delta$-connected components of $E$.
It is clear that $h_E:(0,+\infty)\rightarrow\mathbb{Z}_{\geq 1}$ is non-increasing. Let $\{\delta_k\}_{k\ge 1}$ be the set of discontinuous points of $h_E$ in decreasing order. Then $h_E(\delta)=1$ on $[\delta_1,\infty)$, and is constant on $[\delta_{k+1},\delta_k)$ for $k\ge 1$.
 We call $m_k=h_E(\delta_{k+1})-h_E(\delta_k)$ the \emph{multiplicity} of $\delta_k$ and define the \emph{gap sequence} of $E$, denoted by $\{g_i(E)\}_{i\ge 1}$, to be the sequence
$$
\underbrace{\delta_1,\dots,\delta_1}_{m_1},\underbrace{\delta_2,\dots,\delta_2}_{m_2},
\dots\underbrace{\delta_k,\dots,\delta_k}_{m_k},\dots
$$
In other words,
\begin{equation}\label{gap}
g_i(E)=\delta_k,\quad\text{if } h_E(\delta_k)\le i<h_E(\delta_{k+1}).
\end{equation}

\begin{lem}\label{lem:rem1}
  If two compact  metric spaces  $(E, \rho)$ and $(E', \rho')$   are Lipschitz equivalent, then $g_i(E)\asymp g_i(E')$.
\end{lem}

\begin{proof} The lemma is proved in \cite{RRY08} in the case that both $E$ and $E'$ are subsets of $(\R^d,|\cdot|)$.
Their proof works for metric spaces without any change.
\end{proof}

Miao \emph{et al.} \cite{MXX17} gave the following criterion for power law property in the case that $E\subset \R^d$, but the conclusion and proof work for metric spaces. Here we give an alternative  proof for the reader's sake.

\begin{lem}[\cite{MXX17}]\label{MXX}
Let $(E, \rho)$ be a compact metric space and let $\gamma>0$. Then
$$
g_i(E)\asymp i^{-1/\gamma}\Leftrightarrow h_E(\delta)\asymp\delta^{-\gamma}.
$$
\end{lem}

\begin{proof} Let $\{\delta_k\}_{k\ge 1}$ be the set of discontinuous points of $h_E$ in decreasing order.
First, we show that either
$g_i(E)\asymp i^{-1/\gamma}, i\geq 1 $ or $h_E(\delta)\asymp\delta^{-\gamma}, \delta\in (0,1)$ will imply that
\begin{equation}\label{bound-gap}
M=\sup_{k\geq 1}\frac{\delta_k}{\delta_{k+1}}<\infty.
\end{equation}
If $g_i(E)\asymp i^{-1/\gamma}, i\geq 1, $ holds, then for any $k\geq 1$, there exists $i$ such
that $g_i(E)=\delta_k$ and $g_{i+1}(E)=\delta_{k+1}$, so \eqref{bound-gap} holds in this case.
If  $h_E(\delta)\asymp\delta^{-\gamma}, \delta\in (0,1),$ holds, then this together with
$\lim_{\delta\to (\delta_{k+1})+} h_E(\delta)=h_E(\delta_k)$ imply \eqref{bound-gap} again.

Finally, using \eqref{gap} and \eqref{bound-gap}, we obtain the lemma by a routine estimation.
\end{proof}


 \begin{remark}\label{equalF}
\emph{
 Let $m>1$ be an integer and $\gamma>0$.
 Since $h_E(\delta)$ is non-increasing, we see that  $h_E(m^{-k})\asymp m^{k\gamma}
 (k\ge 0)$,  implies that
 $E$ satisfies the power law property.
 }
\end{remark}


\section{\textbf{Proof of Theorem \ref{thm:cube}}}

Let $m\ge 2$ be an integer. Let $\bD_F=\{\bj_1,\cdots,\bj_r\}\subset\{0,1,\dots,m-1\}^d$.
For $\bj\in \bD_F$ and $y\in\mathbb{R}^d$, we define $\varphi_\bj(y)=\frac{1}{m}(y+\bj)$,
then $\{\varphi_\bj\}_{\bj\in \bD_F}$ is an IFS.
The unique non-empty compact set $F=F(m,\bD_F)$ satisfying
\begin{equation}\label{F}
F=\underset{\bj\in\bD_F}{\bigcup}\varphi_\bj(F)
\end{equation}
is called a \emph{$d$-dimensional fractal cube}, see \cite{XX10}.

For $\boldsymbol{\sigma}=\sigma_1\dots\sigma_k\in\bD_F^k$, we define $\varphi_{\boldsymbol{\sigma}}(z)=\varphi_{\sigma_1}\circ\dots\circ \varphi_{\sigma_k}(z)$.
We call
$$
F_k=\bigcup_{\boldsymbol{\sigma}\in\bD_F^k}\varphi_{\boldsymbol{\sigma}}([0,1]^d)
$$
the $k$-th approximation of $F$. We call $\varphi_{\boldsymbol{\sigma}}([0,1]^d)$ a \emph{$k$-th basic cube}, and call it a \emph{$k$-th boundary cube}   if in addition $\varphi_{\boldsymbol{\sigma}}([0,1]^d)\cap\partial[0,1]^d\ne\emptyset$.
Clearly, $F_k\subset F_{k-1}$ for all $k\ge 1$ and $F=\bigcap_{k=0}^\infty F_k$.

Recall that $F$ is degenerated if $F$ is contained in a face of $[0,1]^d$ with dimension $d-1$.
A connected component $C$ of $F_k$ is called a \emph{$k$-th island} of $F_k$ if $C\cap\partial[0,1]^d=\emptyset$, see \cite{HR20}.
Huang and Rao (\cite{HR20},Theorem 3.1) proved that

\begin{prop}[\cite{HR20}]\label{hasisland}
Let $F$ be a $d$-dimensional fractal cube which is non-degenerated. Then $F$ has trivial points if and only if there is an integer $p\ge 1$ such that $F_p$ contains an island.
\end{prop}

%

For $A\subset\mathbb{R}^d$,  we denote by $N_c(A)$ the number of connected components of $A$.

\begin{lem}\label{inverse}
Let $F$ be a $d$-dimensional fractal cube defined in \eqref{F}. If $F$ has no trivial point, then for any $\delta\in(0,1)$ we have
\begin{equation}\label{eqlem1}
h_F(\delta)\le c_1\delta^{-\log_m(r-1)},
\end{equation} where $c_1>0$ is a constant.
\end{lem}
\begin{proof}
Since a degenerated fractal cube $F=F(m, \SD_F)$ is always isometric to a non-degenerated fractal cube
$F'=F(m, \SD_{F'})$ such that $\#\SD_F=\#\SD_{F'}$, so we only need to consider the case that $F$ is non-degenerated.

Let $q=\lfloor\log_m\sqrt{d}\rfloor+1$, where $\lfloor a\rfloor$ denotes the greatest integer no larger than $a$.
We will prove
\begin{equation}\label{unstable}
h_F(m^{-k})\le 2d(r-1)^{k+q},\quad\text{for all }k\ge 1
\end{equation}
by induction on $d$. Notice  that $r\ge m$ since $F$ has no trivial point.

If $d=1$, then $r=m$, $F=[0,1]$ and $h_F(m^{-k})=1$ for all $k\geq 1$, so \eqref{unstable} holds in this case.

Assume that  \eqref{unstable} holds for all $d'$-dimensional fractal cubes which have no trivial point, where $d'<d$. Now let $F$ be a $d$-dimensional fractal cube which has no trivial point. Denote the $(d-1)$-faces of $[0,1]^d$ by $\Omega_1,\dots,\Omega_{2d}$. Let
$$
r_i=\#\{j\in\Sigma;~\varphi_j([0,1]^d)\cap\Omega_i\ne\emptyset\}, \\
$$
be the number of $1$-th boundary cube intersecting the face $\Omega_i$.
Since $F$ is non-degenerated, we have $r_i\le r-1$ for all $1\le i\le 2d$.
Clearly, the number of $k$-th boundary cubes which intersect $\Omega_i$ is $r_i^k$,
so the number of $k$-th boundary cubes of $F_k$ are at most $2d(r-1)^k$.
Notice that $F$ has no trivial point, then $F_k$ has no island by Proposition \ref{hasisland}.
Thus each connected component of $F_k$ contains at least one $k$-th boundary cube.
By the choice of $q$ we see that the diameter of a $(k+q)$-th basic cube is less than $m^{-k}$, then the points of $F$ in a connected component of $F_{k+q}$ are contained in a $m^{-k}$-connected component of $F$. Therefore, we have
$$
h_F(m^{-k})\le N_c(F_{k+q}) \le 2d(r-1)^{k+q}.
$$
This completes the proof of  \eqref{unstable}.
Finally, by an argument similar to Lemma \ref{equalF}, we obtain \eqref{eqlem1}.
\end{proof}


\begin{proof}[\textbf{Proof of Theorem \ref{thm:cube}}]
Let $F$ be a fractal cube defined by \eqref{F}.
Notice that $\dim_B F=\log_m r.$
The necessity of the theorem is guaranteed by Lemma \ref{inverse}.

Now we prove the sufficiency. Suppose that $F$ has trivial points.
By Remark \ref{equalF},
it is sufficient to show
\begin{equation}\label{eq:h}
h_F(m^{-k})\asymp m^{k\dim_BF},\quad k\ge 0.
\end{equation}
As before, we only need to consider the case that $F$ is non-degenerated.

By Proposition \ref{hasisland}, $F_p$ contains a $p$-th island $C$ for some $p\geq 1$.
Fix $k>p$. Clearly any $m^{-k}$-connected component of $F$ is contained in a connected component of $F_{k-1}$.
It is easy to see that $\varphi_{\boldsymbol{\sigma}}(C)$ is a $(k-1)$-th island of $F_{k-1}$ for any $\boldsymbol{\sigma}\in\bD_F^{k-p-1}$, and the distance of
any two distinct $(k-1)$-th islands of the form $\varphi_{\boldsymbol{\sigma}}(C)$ is no less than $m^{1-k}$, so
\begin{equation}\label{less}
h_F(m^{-k})\ge N_c(F_{k-1})\geq  (\#\bD_F)^{k-p-1}=r^{k-p-1}.
\end{equation}

Let $q=\lfloor\log_m \sqrt{d}\rfloor+1$. Then the points of $F$ in a $(k+q)$-th basic cube is contained in a $m^{-k}$-connected component of $F$, which implies that
$h_F(m^{-k}) \le r^{k+q}$.
This together with \eqref{less} imply $h_F(m^{-k})\asymp r^k=m^{k\dim_BF}$.
The theorem is proved.
\end{proof}

\begin{remark}
\emph{It is shown in \cite{HR20}   that if a fractal cube $F$ has a trivial point, then the Hausdorff dimension of the collection of its non-trivial points is strictly less than $\dim_HF$.}
\end{remark}

\section{\textbf{Proof of Theorem \ref{thm:law}}}
In this section, we always assume that $E$ is a self-affine Sierpi$\acute{\text{n}}$ski sponge defined in \eqref{Ssponge}.
 We call
\begin{equation}\label{eq-E_k}
E_k=\underset{\bomega\in\bD^k }\bigcup S_\bomega([0,1]^d)
\end{equation}
the \emph{$k$-th approximation} of $E$, and call each $S_\bomega([0,1]^d)$ a \emph{$k$-th basic pillar} of $E_k$.

Recall that   $\pi_j(x_1,\dots, x_d)=(x_1,\dots, x_j)$, $1\le j \le d-1$ and
by convention we set $\#\pi_0(\bD)=1$ and  $\pi_d=id$.
Note that  $\pi_j(E)$ is a self-affine Sierpi$\acute{\text{n}}$ski sponge which determined by $\{n_\ell\}_{\ell=1}^{j}$ and $\pi_j(\bD)$.
By \cite{KP96}, the box-counting dimension of $E$ is
\begin{equation}\label{boxdim}
\dim_BE=\sum_{j=1}^d\frac{1}{\log n_j}\log\frac{\#\pi_j(\bD)}{\#\pi_{j-1}(\bD)},
\end{equation}

Recall that $\Lambda=\text{diag}(n_1,\dots, n_d)$.
 A $k$-th basic pillar of $E$ can be represented by
\begin{equation}\label{eq_pillar}
S_{\omega_1\dots\omega_k}([0,1]^d)
 =\sum_{l=1}^k \Lambda^{-l}\omega_l+\prod_{j=1}^d[0,n_j^{-k}],
\end{equation}
where $\omega_l\dots \omega_k\in\bD^k$.
For $1\le j\le d$, denote $\ell_j(k)=\lfloor k\log n_d/\log n_j\rfloor$.
Clearly $n_j^{-\ell_j(k)}\approx n_d^{-k}$ and $\ell_1(k)\geq \ell_2(k)\geq \cdots \geq \ell_d(k)=k$.
We call
\begin{equation}\label{eq_approx_cube}
Q_k:=\left(\sum_{l=1}^{\ell_1(k)}\frac{i_1(\omega_l)}{n_1^l},\dots,
\sum_{l=1}^{\ell_d(k)}\frac{i_d(\omega_l)}{n_d^l}\right)+\prod_{j=1}^d[0,n_j^{-\ell_j(k)}],
\end{equation}
a \emph{$k$-th approximate box} of $E$, if $\omega_l=(i_1(\omega_l),\dots,i_d(\omega_l))\in\bD$ for  $1\le l\le \ell_1(k)$.
Let $\widetilde{E}_k$ be the union of all $k$-th approximate boxes. It is clear that $\widetilde{E}_k\subset E_k$.

Let $\mu$ be the uniform Bernoulli measure on $E$, that is, every $k$-th basic pillar has measure $1/N^k$ in $\mu$.
The following lemma illustrates a nice covering property of self-affine Sierpi$\acute{\text{n}}$ski sponges;
it is contained implicitly in \cite{KP96} and it is a special case of a result in \cite{HRWX}.

\begin{lem}[\cite{KP96, HRWX}] \label{lem:box} Let $E$ be a self-affine Sierpi$\acute{\text{n}}$ski sponge.
Let $R$ be a $k$-th basic pillar of $E$. Then the number of $(k+p)$-th approximate boxes contained in $R$
is comparable to
$$\frac{n_d^{(k+p)\dim_B E}}{N^k}, \quad p\geq 1.$$
\end{lem}

\begin{coro}\label{cor:upper}
Let $V$ be a union of some $k$-th cylinders of $E$.
 Then there exists a constant $M_1>0$ such that
\begin{equation}\label{eq:V-1}
h_V(\delta)\leq M_1\mu(V)\delta^{-\dim_B E}, \quad \delta\in (0, n_d^{-k}).
\end{equation}
\end{coro}

\begin{proof}
Let $p\geq 1$ and $\delta=n_d^{-k-p}$. By Lemma \ref{lem:box},
 the number of $(k+p)$-th
approximate boxes contained in the union of the corresponding $k$-th basic pillars of $V$ is comparable to $\mu(V)n_d^{(k+p)\dim_B E}$.
Since every $(k+p)$-th approximate  box can intersect a bounded number of $\delta$-connected component of $E$,
we obtain the lemma.
\end{proof}

Similar as Section 3, a connected component $C$ of $E_k$ (resp. $\widetilde{E}_k$) is called a \emph{$k$-th island of $E_k$} (resp. $\widetilde{E}_k$) if $C\cap\partial[0,1]^d=\emptyset$. It is easy to see
that $E_k$ has islands if and only if $\widetilde{E}_k$ has islands.
Recall that  $E$ is said to be degenerated if $E$ is contained in a face of $[0,1]^d$ with dimension $d-1$.
Zhang and Xu \cite[Theorem 4.1]{ZX21}  proved that

\begin{prop}\label{has-island2}
Let $E$ be a non-degenerated  self-affine Sierpi$\acute{\text{n}}$ski sponge.
 Then $E$ has trivial points if and only if there is an integer $q\ge 1$ such that $E_q$ has islands.
\end{prop}

%

Let $Q_k\in \widetilde{E}_k$,
we call $Q_k$ a \emph{$k$-th boundary approximate box} of $E$ if $Q_k\cap \partial[0,1]^d\neq \emptyset.$

Let $W$ be a $k$-th cylinder of $E$. Write $W=f(E)$, then $f([0,1]^d)$ is the corresponding $k$-th
basic pillar of $W$.
A $\delta$-connected component $U$ of $W$ is called an
\emph{inner $\delta$-connected component}, if
 $$
 \dist\left(U, \partial(f([0,1]^d))\right)>\delta,
 $$
otherwise, we call $U$ a \emph{boundary $\delta$-connected component}.
We denote by $h^b_W(\delta)$ the number of boundary $\delta$-connected components of $W$,
and by $h^i_W(\delta)$ the number of inner $\delta$-connected components of $W$.


\begin{lem}\label{lem:key}
Let $E$ be a non-degenerated self-affine Sierpi$\acute{\text{n}}$ski sponge.

\emph{(i)} Let $W$ be a $k$-th cylinder of $E$. Then
\begin{equation}\label{eq:Wb}
h^b_W(\delta)=o(\mu(W)\delta^{-\dim_B E}),\quad \delta\to 0.
\end{equation}

\emph{(ii)} If  $E$ satisfies the maximal power law,
then $E$ possesses trivial points; moreover, there exists an integer $p_0\geq 1$ and a constant $c'>0$  such that for any $k\geq 1$, any $k$-th cylinder $W$ of $E$
and any $\delta\leq n_d^{-(k+p_0)}$,
   $$h^i_W(\delta)\geq  c'h_W(\delta).$$
\end{lem}
\begin{proof} (i)
Write $W=f(E)$. Let $W_p^*$ be the union of $(k+p)$-th cylinders of $W$ whose corresponding basic pillars
 intersecting the boundary of $f([0,1]^d)$. Since $E$ is non-degenerated, we have
\begin{equation}\label{eq:Wp}
\mu(W_p^*) =o(\mu(W)), \quad p\to \infty.
\end{equation}
 Set $\delta=n_d^{-(k+p)}$. If $U$ is a boundary $\delta$-connected component of $W$,
 then $U\cap W_p^*\neq \emptyset$.
Therefore,
 $$
 h^b_{W}(\delta)\leq h_{W_p^*}(\delta) \leq M_1\mu(W_p^*)
 \delta^{-\dim_B E},
 $$
where the last inequality is due to Corollary \ref{cor:upper}.
  This together with \eqref{eq:Wp} imply \eqref{eq:Wb}.

(ii)
The assumption that $E$ satisfies the maximal power law implies that there exists $M_2>0$ such that
 \begin{equation}\label{eq:V-2}
 h_E(\delta)\geq  M_2\cdot \delta^{-\dim_B E},\quad \text{for any}~~ \delta\in(0,1).
 \end{equation}
 If $E$ does not possess trivial points, then for all $k\geq 1$,
 $E_k$ does not contain any $k$-th island by Proposition \ref{has-island2},
 furthermore, $\widetilde{E}_k$ does not contain any $k$-th island.
 Thus each $\delta$-connected component of $E$ contains points of $E\cap \partial([0,1]^d)$,
 so $h_E(\delta)=h^b_E(\delta)$.
 On the other hand, by (i) we have $h^b_E(\delta)=o(\delta^{-\dim_B E})$ as $\delta\rightarrow 0$,
 which contradicts \eqref{eq:V-2}. This proves that $E$ possesses trivial points.


 Let $W$ be a $k$-th cylinder of $E$.
 Since $h_{A\cup B}(\delta)\leq h_A(\delta)+h_B(\delta)$,  by \eqref{eq:V-2} we have
\begin{equation}\label{eq:V333}
 h_W(\delta)\geq N^{-k}h_E(\delta)\geq M_2\mu(W) \delta^{-\dim_B E}, \quad \text{for any}~\delta\in (0,1).
\end{equation}
Take $\varepsilon\leq M_2/2$. By \eqref{eq:Wb}, there exists an integer $p_0\geq 1$ such that
 $$
 h^b_W(\delta)\leq \varepsilon \mu(W)\delta^{-\dim_B E} \text{ for } \delta\leq n_d^{-(k+p_0)}.
 $$
   This together with \eqref{eq:V333} and \eqref{eq:V-1} imply that
   for $\delta\leq n_d^{-(k+p_0)}$,
 $$
 h^i_{W}(\delta)=h_W(\delta)-h^b_W(\delta)\geq \frac{M_2}{2} \mu(W)\delta^{-\dim_B E}\geq c'h_{W}(\delta),
 $$
 where $c'=M_2/(2M_1)$ and $M_1$ is the constant in Corollary \ref{cor:upper}.
 The lemma is proved.
\end{proof}


%
%

\medskip
\begin{proof}[\textbf{Proof of Theorem \ref{thm:law}}]
We will prove this theorem by induction on $d$. Notice that $E$ is a $1$-dimensional fractal cube if $d=1$,
 and in this case the theorem holds by Theorem \ref{thm:cube}.
Assume that the theorem holds for all $d'$-dimensional self-affine Sierpi$\acute{\text{n}}$ski sponge with $d'\le d-1$.

Now we consider the $d$-dimensional self-affine Sierpi$\acute{\text{n}}$ski sponge $E$.
Denote $G:=\pi_{d-1}(E)$.  First,  $G$ is non-degenerated since $E$ is.  Secondly, by \eqref{boxdim} we have
\begin{equation}\label{eq:dimG}
\log_{n_d} \frac{N}{\#\pi_{d-1}(\bD)}+\dim_B G=\dim_B E.
\end{equation}

``$\Leftarrow$": Suppose $E$ and all $\pi_j(E)~(1\leq j \leq d-1)$ possess trivial points,
then $G=\pi_{d-1}(E)$ satisfies the maximal power law by induction hypothesis.
We will show that $E$ satisfies the maximal power law.

Firstly,
by Corollary \ref{cor:upper},
\begin{equation}\label{uppbound1}
h_E(\delta)\leq M_1\delta^{-\dim_B E} \quad \text{ for all }\delta\in (0,1).
\end{equation}
Now we consider the lower bound of $h_E(\delta)$.

Since $E$ possesses trivial points, by Proposition \ref{has-island2}, there exists an integer $q_0\geq 1$ such that
$E_{q_0}$ has a $q_0$-th island, which we denote by $I$.
Let $p_0$ be the constant in Lemma \ref{lem:key} (ii).
Let $k\geq q_0$  and $\delta=n_d^{-(k+p_0)}$.
Since $G$ satisfies the maximal power law, there exists  a constant $c>0$ such that
\begin{equation}\label{uppbound_4}
c^{-1}\delta^{-\text{dim}_B G} \leq h_G(\delta)\leq c\delta^{-\text{dim}_B G}.
\end{equation}

It is easy to see that
$S_{\boldsymbol{\tau}}(I)$ is a $k$-th island of $E$ for any $\boldsymbol{\tau}\in \SD^{k-q_0}$.
So $E_k$ has $N^{k-q_0}$ number of $k$-th islands like $I':=S_\bomega(I)$ for some $\bomega\in \SD^{k-q_0}$.
Since the distance of any two $k$-th islands of $E_k$ is at least $n_d^{-k}$,
we have
\begin{equation}\label{lowerbound_1}
h_E(\delta)\geq N^{k-q_0}\cdot h_{E\cap I'}(\delta).
\end{equation}
Let $W$ be any $k$-th cylinder of $E$ contained in $I'$. Since $\pi_{d-1}$ is contractive,
  we obtain
\begin{equation}\label{lowerbound_a}
  h_{E\cap I'}(\delta)\geq   h_{G\cap \pi_{d-1}(I')}( \delta)\geq c'  h_{ \pi_{d-1}(W)}(\delta),
\end{equation}
 where   the last inequality is due to Lemma \ref{lem:key}(ii) with the constant $c'$ depends on $G$.
  Furthermore, from $h_{A\cup B}(\delta)\leq h_A(\delta)+h_B(\delta)$ we infer that
\begin{equation}\label{lowerbound_2}
 (\#\pi_{d-1}(\bD))^k \cdot h_{\pi_{d-1}(W)}(\delta)\geq h_G(\delta).
 \end{equation}
By (\ref{uppbound_4}-\ref{lowerbound_2}), we have
\begin{eqnarray*}
h_E(\delta)&\geq& N^{k-q_0} \cdot \frac{c'h_G(\delta)}{(\#\pi_{d-1}(\bD))^k}
\geq c'c^{-1} N^{-q_0}\cdot
\left(\frac{ N}{\#\pi_{d-1}(\bD)}\right)^k\cdot \delta^{-\text{dim}_B G}\\
&=& c'c^{-1} N^{-q_0} \cdot \left(\frac{\#\pi_{d-1}(\bD)}{N}\right)^{p_0} \cdot \delta^{-\text{dim}_B E},
\end{eqnarray*}
where the last equality holds by \eqref{eq:dimG}.
This together with \eqref{uppbound1} imply that $E$  satisfies the maximal power law.

 ``$\Rightarrow$": Suppose $E$ satisfies the maximal power law.  Then  $E$  possesses  trivial points by Lemma \ref{lem:key} (ii).
So it is sufficient to show that $G=\pi_{d-1}(E)$ satisfies the maximal power law by induction hypothesis.

Suppose on the contrary this is false.
  Then
 given $\varepsilon>0$, there exists $\delta$ as small as we want, such that
 \begin{equation}\label{ineq-h_G}
 h_G(\delta)\leq\varepsilon \delta^{-\dim_B G}.
 \end{equation}

Let $W$ be a $k$-th cylinder of $E$,
then $V:=\pi_{d-1}(W)$ is a $k$-th cylinder of $G$.
Let $\mu'$ be the uniform Bernoulli measure on $G$.
By Lemma \ref{lem:key} (i), there exists an integer  $p_1\geq 1$  such that
$$
h^b_V(\eta)\leq \varepsilon \mu'(V) \eta^{-\dim_B G} \text{ for } \eta\leq n_{d-1}^{-(k+p_1)}.
$$
We choose  $\delta$ small and thus  $k$ large so that
$n^{-(k+1)}\leq \delta<n_d^{-k}<n_{d-1}^{-(k+p_1)}$, then
$$h^b_{V}(\delta)  \leq \frac{\varepsilon
 \delta^{-\dim_B G}}{(N')^{k}},$$
 where $N'=\#\pi_{d-1}(\SD)$.
On the other hand, by \eqref{ineq-h_G},
$$
(N')^k h^i_{V}(\delta)\leq h_G(\delta)\leq \varepsilon \delta^{-\dim_B G}.
$$
So we obtain
\begin{equation}\label{eq:h_V}
h_V(\delta)=h^b_{V}(\delta)+h^i_{V}(\delta)\leq
\frac{2\varepsilon  \delta^{-\dim_B G}}{(N')^{k}}.
\end{equation}

Now we estimate $h_W(\delta')$, where
 $\delta'=\sqrt{2}n_d^{-k}> \sqrt{\delta^2+(n_d^{-k})^2}$.
Since $W$ is contained in $\pi_{d-1}(W)\times [b, b+n_d^{-k}]$ for some $b\in [0,1]$, we deduce that
if  two points  $\pi_{d-1}(x)$ and $\pi_{d-1}(y)$  belong to a same $\delta$-connected component of $G$,
then $x$ and $y$ belong to a same $\delta'$-connected component of $E$.
  Therefore,
\begin{equation}\label{eq:W-2}
  h_W(\delta')\leq h_{V}(\delta),
\end{equation}
and consequently,
$$
h_{E}(\delta')  \leq N^k\cdot  h_W(\delta') \leq N^k\cdot \frac{2\varepsilon  \delta^{-\dim_B G}}{(N')^{k}}
 \leq  M'\varepsilon (\delta')^{-\dim_B E}
$$
for $M'=2(\sqrt{2})^{\dim_B E}n_d^{\dim_B G}$.
This is a contradiction since $E$ satisfies the maximal power law.
The theorem is proved.
%
%
\end{proof}

\section{\textbf{Proof of Theorem \ref{thm:perfect}}}


Before proving Theorem \ref{thm:perfect}, we prove a  finite type property of totally disconnected self-affine Sierpi$\acute{\text{n}}$ski sponge. The proof is similar to \cite{XX10} and \cite{MXX17}, which dealt with
 fractal cubes and Bedford-McMullen carpets, respectively.

\begin{thm}\label{finite-type}
Let $E$ be a totally disconnected self-affine Sierpi$\acute{\text{n}}$ski sponge, then  there is an integer $M_3>0$ such that
for every integer $k\geq 1$, each connected component of $E_k$ contains at most $M_3$ $k$-th basic pillars.
\end{thm}



Denote $d_H(A,B)$ the \emph{Hausdorff metric} between two subsets $A$ and $B$ of $\mathbb{R}^d$.
The following lemma is obvious, see for instance \cite{XX10}.

\begin{lem}\label{pre_lem_1}
Suppose $\{X_k\}_{k\ge 1}$ is a collection of connected compact subsets of $[0,1]^d$.
Then there exist a subsequence $\{k_i\}_{i\ge 1}$ and a connected compact set $X$
such that $X_{k_i}\stackrel{d_H}\longrightarrow X$ as $i\rightarrow \infty$.
\end{lem}



\begin{proof}[\textbf{Proof of Theorem \ref{finite-type}}]
Let $E$ be a totally disconnected self-affine Sierpi$\acute{\text{n}}$ski sponge.
We set
$$
\Xi_k=\underset{h\in\{-1,0,1\}^d}{\bigcup} \left({E_k}+h\right).
$$
First, we claim that  there exists an
integer $k_0$ such that for any connected component $X$ of $\Xi_{k_0}$
with $X\cap[0,1]^d\neq\emptyset$, $X$ is contained in $(-1,2)^d$.

Suppose on the contrary that for any $k$ there are connected components $X_k\subset \Xi_k$ and points $x_k\in[0,1]^d\cap X_k$ and $y_k\in\partial[-1,2]^d\cap X_k$.
By Lemma \ref{pre_lem_1}, we can take a subsequence $\{k_i\}_i$ such that $x_{k_i}\rightarrow x^*\in[0,1]^d,
y_{k_i}\rightarrow y^*\in \partial[-1,2]^d$ and $X_{k_i}\stackrel{d_H}\longrightarrow X$ for some connected compact set $X$ with
$$X\subset \underset{h\in\{-1,0,1\}^d}\bigcup (E+h), \ x^*\in X\cap [0,1]^d, \text{ and } y^*\in X\cap\partial[-1,2]^d.$$
Clearly $\underset{h\in\{-1,0,1\}^d}\bigcup (E+h)$ is totally disconnected
 since it is a finite union
of totally disconnected compact sets. (A carefully proof of this fact can be found in  \cite{XX10}).
This  contradiction proves our claim.

For any $k\ge 1$, let $U$ be a connected component of $E_{k+k_0}$ such that $U\cap S_\bomega([0,1]^d)\ne\emptyset$ for some $\bomega\in\bD^k$. Notice that $S_\bomega([-1,2]^d)\cap E_{k+k_0}\subset S_\bomega(\Xi_{k_0})$,
then $[-1,2]^d\cap S_\bomega^{-1}(U)\subset \Xi_{k_0}$.
By the claim above, every connected component of $[-1,2]^d\cap S_\bomega^{-1}(U)$ which intersects $[0,1]^d$ is contained in $(-1,2)^d$.
This implies that
$$S_\bomega^{-1}(U)  \subset (-1,2)^d$$
since $U$ is connected.
It follows that $U\subset S_\bomega((-1,2)^d)$, so $U$ contains at most $3^dN^{k_0}$ $(k+k_0)$-th basic pillars
of $E$. Thus $E$ is of finite type.
\end{proof}

Recall that $\pi_j(x_1,\dots,x_d)=(x_1,\dots, x_j)$ for $1\leq j \leq d-1$.
Denote $\hat \pi_d(x_1,\dots, x_d)=x_d$. Let $z_1,\dots, z_p\in E$.
For $0< \delta<1$, we call $z_1,\dots, z_p$ a \emph{$\delta$-chain} of $E$ if $|z_{i+1}-z_i|\leq \delta$ for $1\leq i \leq p-1$,
and define the \emph{size} of the chain as
$$
L=\frac{\text{diam} \{z_1,\dots, z_p\}}{\delta}.
$$

\begin{proof}[\textbf{Proof of Theorem \ref{thm:perfect}}]
We will prove the theorem by induction on $d$.
If $d=1$, $E$ is a $1$-dimensional fractal cube, and the theorem holds by Remark \ref{perfectly_prop}.

\medskip

 \emph{Totally disconnected $\Rightarrow$ perfectly disconnected:}
  Suppose that $E$ and $\pi_j(E)$, $j=1,\dots, d-1$,
   are totally disconnected. By induction hypothesis, $\pi_{d-1}(E)$
   is perfectly disconnected.
   We will show that $E$ is perfectly disconnected.

Take $\delta\in(0,1)$.  Let $U$ be a $\delta$-connected component of $E$. Let $k$ be the
 integer such
   that $n_d^{-(k+1)}\leq \delta <n_d^{-k}$. Then $U$ is contained in a connected component $V$ of $E_k$.
   Let $M_3$ be a constant such that Theorem \ref{finite-type} holds for $E$ and $\pi_{d-1}(E)$ simultaneously,
   then $V$ contains at most $M_3$ $k$-th basic pillars.
   Denote $L=\text{diam}(U)/\delta$.

   Let $z_1,\dots, z_p$ be a $\delta$-chain in $U$ such that $|z_1-z_p|\geq L\delta/2$.
    Then $\pi_{d-1}(z_1),\dots, \pi_{d-1}(z_p)$ is also a $\delta$-chain
  in $\pi_{d-1}(E)$.  Since
    $|\hat \pi_d(x)-\hat\pi_d(y)|\leq M_3 n_d^{-k}$ for any $x,y\in V$,
 we have
    $$|\pi_{d-1}(z_1)-\pi_{d-1}(z_p)|\geq \sqrt{L^2\delta^2/4-M_3^2n_d^{-2k}}\geq
   n_d^{-k} \sqrt{\frac{L^2}{4n_d^2}-M_3^2} ,$$
   so
   $$
   L'=\frac{|\pi_{d-1}(z_1)-\pi_{d-1}(z_p)|}{\delta}> \sqrt{\frac{L^2}{4n_d^2}-M_3^2}.
   $$

  If $E$ is not perfectly disconnected, then we can choose $U$ such that $L$ is arbitrarily large,
 so $\diam(\pi_{d-1}(U))/\delta\geq L'$ can  be arbitrarily large,
 which contradicts
 the fact that $\pi_{d-1}(E)$ is perfectly disconnected.

 \medskip

  \emph{Perfectly disconnected $\Rightarrow$ totally disconnected:}
  Assume that $E$ is perfectly disconnected.
  Clearly $E$ is totally disconnected. Hence, by induction hypothesis,
  we only need to show that  $G:=\pi_{d-1}(E)$ is perfectly disconnected.

   Firstly, we claim that $G$ must be  totally disconnected.

    Suppose on the contrary that the claim is false.
    Let $\Gamma$ be a connected subset of $(0,1)^{d-1}\cap G$. Fix two points $a, b\in \Gamma$.
    Take any $\bomega\in (\pi_{d-1}(\SD))^k$, then $S'_\bomega(\Gamma)$ is contained
    in the interior of the $k$-th basic pillar $V=S'_\bomega([0,1]^{d-1})$,
    where $\{S'_\bi\}_{\bi\in \pi_{d-1}(\SD)}$ is the IFS of $G$.
    Denote $a'=S'_\bomega(a)$, $b'=S'_\bomega(b)$.
    Then $|b'-a'|\geq |b-a|/n_{d-1}^k$.

    Let $\delta=n_{d}^{-k}$. Let $z_1=a', \dots, z_p=b'$ be a $\delta$-chain in $S'_\bomega(\Gamma)\subset V$. Let $W$ be a $k$-th cylinder of $E$ such that $\pi_{d-1}(W)=V\cap G$.
    Let $y_1,\dots, y_p$ be a sequence in $W$ such that $\pi_{d-1}(y_j)=z_j,~1\leq j \leq p$, then it is a $(\sqrt{2}\delta)$-chain.
    Since
    $$
    L=\frac{|y_1-y_p|}{\sqrt{2}\delta}\geq \frac{|z_1-z_p|}{\sqrt{2}\delta}
    \geq \frac{|b-a|}{\sqrt{2}} \cdot\left(\frac{n_d}{n_{d-1}}\right)^k
    $$
    can be arbitrarily large when $k$ tends to $\infty$, we conclude that $E$ is not perfectly disconnected,
    which is a contradiction. The claim is proved.

 Secondly, suppose on the contrary that $G$ is not perfectly disconnected.

Take $\delta\in(0,1)$.  Let $U$ be a $\delta$-connected component of $G$.  Denote $L=\text{diam}(U)/\delta$. Let $k$ be the
 integer such
   that $n_d^{-(k+1)}\leq \delta <n_d^{-k}$.
    Then $U$ is contained in a connected component $V$ of $G_k$.
   Since $G$ is totally disconnected, by Theorem \ref{finite-type},
   $V$ contains at most $M_3$ $k$-th  basic pillars of $G$,
   and we denote the corresponding $k$-th cylinders by $V_1,\dots, V_h$ where $h\leq M_3$.

Let $z_1,\dots, z_p$ be a $\delta$-chain of $U$ such that $|z_1-z_p|=L\delta$.
   Without loss of generality, we may assume that the diameter of $\{z_1,\dots, z_p\}\cap V_j$ attains
   the maximality when $j=1$.
Let $\{x_j\}_{j=1}^\ell$ be the subsequence of $\{z_j\}_{j=1}^p$ located in $V_1$, then
    $|x_1-x_\ell|\geq L\delta/h$.
    Let $\Delta=\underset{1\le j\le \ell-1}{\max}|x_j-x_{j+1}|$,
   then  $x_1,\dots, x_\ell$ is a $\Delta$-chain in $V_1$.

   Let $W$ be a $k$-th cylinder of $E$ such that $\pi_{d-1}(W)=V_1$.
   Then the pre-image of $x_j$ in $W$, which we denote by $y_1,\dots ,y_\ell$,
   is a $\sqrt{\Delta^2+(n_d\delta)^2}$-chain. The size of this chain is
   $$
   \widetilde{L}\geq \frac{|y_1-y_\ell|}{\sqrt{\Delta^2+(n_d\delta)^2}}\geq
   \frac{|x_1-x_\ell|}{\sqrt{\Delta^2+(n_d\delta)^2}}\geq
   \frac{L\delta}{h\sqrt{\Delta^2+(n_d\delta)^2}}.
   $$

    \textit{Case 1.} $\Delta\leq \sqrt{L}\delta$.

   In this case we have $\widetilde{L}\geq \frac{L}{M_3\sqrt{L+n_d^2}}$. Since $G$ is not perfectly disconnected,
     $\widetilde{L}$ can be arbitrarily large when $L\to \infty$, we deduce that $E$ is not perfectly disconnected, which is a contradiction.

     \textit{Case 2.}  $\Delta> \sqrt{L}\delta$.

     Let $1\leq j^*\leq \ell-1$ be the index such that $|x_{j^*+1}-x_{j^*}|=\Delta$.
  Denote   the sub-chain of $z_1,\dots,z_p$ between $x_{j^*}$ and $x_{j^*+1}$ by
     $$z_1'=z_{m+1},\dots, z'_s=z_{m+s}.$$
     Then $(z'_j)_{j=1}^s$ belong to $V_2\cup \cdots \cup V_h$ and  $|z'_1-z'_s|\geq (\sqrt{L}-2)\delta$.

  Now we  repeat the above argument by considering the $\delta$-chain
  $(z'_j)_{j=1}^s$ in $V_2\cup \cdots \cup V_h$.
 In at most $M_3-1$ steps, we will obtain a $\Delta'$-chain in $E$ with arbitrarily large size when $L\to \infty$.
This finishes the proof of Case 2 and the theorem is proved.
\end{proof}

%


\end{document}